\documentclass{amsart}

\usepackage{amssymb}
\usepackage[margin=1in]{geometry}
\usepackage{url}

\theoremstyle{plain}
    \newtheorem{thm}{Theorem}[section]
    \newtheorem{prp}[thm]{Proposition}
\theoremstyle{remark}
    \newtheorem*{ack}{Acknowledgements}

\newcommand{\cp}{\mathrel{\sim_{\text{p}}}}
\newcommand{\inv}{^{-1}}                % inverse

\title{The transitivity of primary conjugacy in a class of semigroups}

\author{Maria Borralho}
\address{Universidade Aberta, R. Escola Polit\'{e}cnica, 147\\1269-001 Lisboa, Portugal}
\address{CEMAT, Universidade de Lisboa, Av. Rovisco Pais, 1\\1049-001 Lisboa, Portugal}

\begin{document}

\begin{abstract}
Elements $a,b$ of a semigroup $S$ are said to be \emph{primarily conjugate}
or just \emph{p-conjugate}, if there exist $x,y\in S^1$ such that $a=xy$ and $b=yx$.
The p-conjugacy relation generalizes conjugacy in groups, but for general semigroups,
it is not transitive. Finding the classes of semigroups in which this
notion is transitive is an open problem. The aim of this note is to show
that for semigroups satisfying $xy\in\left\{ yx,\left(xy\right)^n\right\}$ for
some $n > 1$, primary conjugacy is transitive.
\end{abstract}

\maketitle

\section{Introduction}

By a notion of conjugacy for a class of semigroups, we mean an equivalence
relation defined in the language of that class of semigroups such that when
restricted to groups, it coincides with the usual notion of conjugacy.

Before introducing the notion of conjugacy that will occupy us, we recall some
standard definitions and notation (we generally follow \cite{Ho95}). For a
semigroup $S$, we denote by $S^1$ the semigroup $S$ if $S$ is a monoid;
otherwise $S^1$ denotes the monoid obtained from $S$ by adjoining an identity
element $1$.

Any reasonable notion of semigroup conjugacy should coincide in groups
with the usual notion. Elements $a,b$ of a group $G$ are conjugate if there exists
$g\in G$ such that $a=g\inv bg$. Conjugacy in groups has several equivalent formulations
that avoid inverses, and hence generalize syntactically to any semigroup. For many
of these notions including the one we focus on here, we refer the reader to 
\cite{ArKiKoMa15,Ko18,KuMa09}.

For example, if $G$ is a group, then
$a,b\in G$ are conjugate if and only if $a=uv$ and $b=vu$ for some
$u,v\in G$. Indeed, if $a = g\inv bg$, then setting $u=g\inv b$ and $v=g$
gives $uv=a$ and $vu=b$; conversely, if $a=uv$ and $b=vu$ for some $u,v\in G$,
then setting $g = v$ gives $g\inv bg = v\inv vuv = uv = a$.

This last formulation was used to define the following relation on a free semigroup $S$
(see \cite{La79}):
\[
a\cp b\qquad\iff\qquad \exists_{u,v\in S^1}\quad a=uv\text{ and }b=vu.
\]
If $S$ is a free semigroup, then $\cp$ is an equivalence relation on $S$
\cite[Cor.5.2]{La79}, and so it can be considered as a notion of conjugacy in
$S$. In a general semigroup $S$, the relation $\cp$ is reflexive and
symmetric, but not transitive. If $a\cp b$ in a semigroup, we say that $a$
and $b$ are \emph{primarily conjugate} or just p-conjugate for short
(hence the subscript in $\cp$); $a$ and $b$ were said to be 
``primarily related'' in \cite{KuMa09}. Lallement \cite{La79} credited the idea
of the relation $\cp$ to Lyndon and Sch\"{u}tzenberger \cite{LySc62}. 

In spite of its name, $\cp$ is a valid notion of conjugacy only in the class of
semigroups in which it is transitive. Otherwise, the transitive closure $\cp^{\ast}$
of $\cp$ has been defined as a conjugacy relation in a general semigroup \cite{Hi06, KuMa07,
KuMa09}. Finding classes of semigroups in which $\cp$ itself is transitive, that is,
$\cp = \cp^{\ast}$, is an open problem. The aim of this note is to prove the
following theorem.

\begin{thm}\label{Thm:main}
Let $n > 1$ be an integer and let $S$ be a semigroup satisfying the following:
for all $x,y\in S$,
\[
xy\in \left\{ yx,\left(xy\right)^n \right\}\,.
\]
Then primary conjugacy $\cp$ is transitive in $S$.
\end{thm}

There are various motivations for studying this particular class of semigroups.
First, it naturally generalizes two classes of semigroups in which $\cp$ is transitive.

\begin{prp}
Let $S$ be a semigroup. 
\begin{enumerate}
  \item If $S$ is commutative, then $\cp$ is transitive.
  \item If $S$ satisfies $xy = (xy)^2$ for all $x,y\in S$, then $\cp$ is transitive.
\end{enumerate}
\end{prp}
\begin{proof}
(1) In a commutative semigroup, $\cp$ is the identity relation and hence it
is trivially transitive.

(2) If $a\cp b$, then $a = uv$ and $b = vu$ for some $u,v\in S^1$. Thus 
$a^2 = (uv)^2 = uv = a$ and $b^2 = (vu)^2 = vu = b$ so that $a,b$ are idempotents.
In particular, $a,b$ are completely regular elements of $S$. The restriction of 
$\cp$ to the set of completely regular elements is a transitive relation 
\cite{Ku06}.
\end{proof}

The other motivation for studying this class of semigroups is that it has
been of recent interest in other contexts. In particular, J.~P.~Ara\'{u}jo and
Kinyon \cite{JPArKi} showed that a semigroup satisfying $x^3 = x$ and 
$xy\in \{yx,(xy)^2\}$ for all $x,y$ is a semilattice of rectangular bands
and groups of exponent $2$. 
\bigskip

The proof of Theorem \ref{Thm:main} was found by first proving the
special cases $n = 2,3,4$ using the automated theorem prover \texttt{Prover9}
developed by McCune \cite{McCune}. After studying these proofs, the pattern
became apparent, leading to the proof of the general case. Note that 
\texttt{Prover9} and other automated theorem provers usually cannot handle
statements like our theorem directly because there is not a way
to specify that $n$ is a fixed positive integer. Thus the approach of
examining a few special cases and then extracting a human proof of the
general case is the most efficient way to use an automated theorem prover
in these circumstances.

\begin{proof}[Proof of Theorem~\ref{Thm:main}]
Suppose $a,b,c\in S$ satisfy $a\cp b$ and $b\cp c$. Since $a\cp b$,
there exist $a_1,a_2\in S^1$ such that $a = a_1 a_2$ and $b = a_2 a_1$.
Similarly, since $b\cp c$, there exist $b_1,b_2\in S^1$ such that
$b = b_1 b_2$ and $c = b_2 b_1$. We want to prove there exist $x,y\in S^1$ such
that $a = xy$ and $c = yx$. If $a = b$ or if $b = c$, then there is nothing to
prove. Thus we may assume without loss of generality that $a_1 a_2\neq a_2 a_1$ and
$b_2 b_1\neq b_1 b_2$.

Assume first that $n = 2$. Then 
\[
a = a_1 a_2 = (a_1 a_2)(a_1 a_2) = a_1 (a_2 a_1) a_2 = a_1 b a_2 = (a_1 b_1)(b_2 a_2)\,,
\]
and
\[
c = b_2 b_1 = (b_2 b_1)(b_2 b_1) = b_2 (b_1 b_2) b_1 = b_2 b b_1 = (b_2 a_2)(a_1 b_1)\,.
\]
Thus setting $x = a_1 b_1$ and $y = b_2 a_2$, we have $a\cp c$ in this case.

Now assume $n > 2$. We have
\begin{align*}
a  &= a_1 a_2 = (a_1 a_2)^n =
\underset{n}{\underbrace{(a_1 a_2)\cdots (a_1 a_2)}} \\
&= a_1 \underset{n-1}{\underbrace{(a_2 a_1)\cdots (a_2 a_1)}} a_2 \\
&= a_1 b^{n-1} a_2 \\
&= a_1 bb^{n-2} a_2 \\
&= a_1 (b_1 b_2) b^{n-2} a_2 \\
&= (a_1 b_1)(b_2 b^{n-2} a_2)
\end{align*}
and
\begin{align*}
c &= b_2 b_1 = (b_2 b_1)^n =
\underset{n}{\underbrace{(b_2 b_1)\cdots (b_2 b_1)}} \\
&= b_2 \underset{n-1}{\underbrace{(b_1 b_2)\cdots (b_1 b_2)}} b_1 \\
&= b_2 b^{n-1} b_1 \\
&= b_2 b^{n-2}b b_1 \\
&= b_2 b^{n-2} (a_2 a_1) b_1 \\
&= (b_2 b^{n-2} a_2)(a_1 b_1)\,.
\end{align*}
Thus setting $x = a_1 b_1$ and $y = b_2 b^{n-2} a_2$, we have that $a\cp c$.
\end{proof}

\begin{ack}
We are pleased to acknowledge the use of the automated theorem prover \texttt{Prover9}
developed by McCune \cite{McCune}. 
We also thank Prof. Jo\~{a}o Ara\'{u}jo for suggesting this problem to us. This paper
forms a part of the author's dissertation in the PhD Program in Computational Algebra
at Universidade Aberta in Portugal.
\end{ack}

\end{document}